\documentclass[reqno,11pt]{amsart}
\usepackage{amsmath,amssymb,amsfonts,amsthm, amscd,indentfirst}
\usepackage{amsmath,latexsym,amssymb,amsmath,
	amscd,amsthm,amsxtra,mathrsfs}
 
\usepackage[top=3cm, bottom=3.5cm, right=2.5cm, left=2.2cm]{geometry}

\usepackage[hyphens]{url} 
\usepackage{hyperref}
\usepackage{bookmark}
\usepackage{mathtools}
\usepackage{etoolbox}
\allowdisplaybreaks 
\mathtoolsset{showonlyrefs=true}
\DeclareRobustCommand{\textsection}{\ifmmode\mathsection\else\S\fi}

\usepackage[
  backend=biber,
  style=alphabetic,
  sorting=nyt,
  maxbibnames=99,
  giveninits=true,
]{biblatex}
\renewbibmacro{in:}{}

\DeclareFieldFormat[article,inproceedings,incollection,unpublished]{title}{#1}
\DeclareFieldFormat[article,inproceedings,incollection,unpublished]{citetitle}{#1}

\DeclareSourcemap{%
  \maps[datatype=bibtex]{%
    \map{\step[fieldsource=shortjournal, fieldtarget=journaltitle]}%
  }%
}

\makeatletter
\newcommand*\bbx@lasthash{}
\newtoggle{bbx@dashed}
\AtBeginBibliography{\global\let\bbx@lasthash\empty}
\AtEveryBibitem{%
  \global\togglefalse{bbx@dashed}%
  \iffieldundef{fullhash}
    {\global\let\bbx@lasthash\empty}
    {\iffieldequalstr{fullhash}{\bbx@lasthash}{\global\toggletrue{bbx@dashed}}{}%
     \xdef\bbx@lasthash{\thefield{fullhash}}}%
}
\DeclareNameFormat{labelname}{%
  \ifboolexpr{togl {bbx@dashed}}
    {\ifnum\value{listcount}=1\relax\bibnamedash\fi}
    {\nameparts{#1}%
     \usebibmacro{name:family-given}
       {\namepartfamily}
       {\namepartgiven}
       {\namepartprefix}
       {\namepartsuffix}%
     \usebibmacro{name:andothers}}%
}
\makeatother

\newcounter{mparcnt}

\usepackage{fancyhdr}
\usepackage{esint}
\usepackage{enumerate}
\usepackage{xcolor}
\usepackage{pictexwd,dcpic}
\usepackage{graphicx}
\usepackage{caption}
\usepackage{graphicx}
\usepackage{caption}
\usepackage{slashed}
\usepackage{epsfig,here}
\usepackage{subfigure,here}

\newtheorem{theorem}{Theorem}[section]
\newtheorem{lemma}[theorem]{Lemma}
\newtheorem{proposition}[theorem]{Proposition}
\newtheorem{corollary}[theorem]{Corollary}

\newcommand{\abs}[1]{\lvert#1\rvert}
\newcommand{\Abs}[1]{\left\lvert#1\right\rvert}

\newcommand{\norm}[1]{\lVert#1\rVert}

\newcommand{\rd}{{\rm d}}
\newcommand{\rdV}{{\rm dV}}
\newcommand{\rVol}{{\rm Vol}}

\newcommand{\rid}{{\rm id}}
\newcommand{\D}{{\slashed{D}}}

\newcommand{\curl}{{\rm curl}\,}
\newcommand{\ip}{\lrcorner\,}
\newcommand{\w}{\wedge}
\def\<{\langle}
\def\>{\rangle}
\def\S{\mathbb{S}}
\def\R{\mathbb{R}}

\newcommand{\ra}{\rightarrow}

\newcommand{\KN}{\mathbin{\bigcirc\mspace{-15mu}\wedge\mspace{3mu}}}

\newcommand{\eq}[1]{\begin{equation}\allowdisplaybreaks\begin{alignedat}{2} #1 \end{alignedat}\end{equation}}

\numberwithin{equation} {section}
\begin{document}
\title[The curl--Yamabe problem]{The curl--Yamabe problem}
\date{\today}
\author{Guofang Wang}\address{ Albert-Ludwigs-Universit\"at Freiburg,
Mathematisches Institut,
Ernst-Zermelo-Str. 1,
D-79104 Freiburg, Germany}
\email{guofang.wang@math.uni-freiburg.de}

\author{Mingwei Zhang}\address{ Wuhan University, School of Mathematics and Statistics, 430072 Wuhan, China and 
Albert-Ludwigs-Universit\"at Freiburg,
Mathematisches Institut,
Ernst-Zermelo-Str. 1,
D-79104 Freiburg, Germany}
\email{zhangmwmath@whu.edu.cn}

\begin{abstract}
We introduce and study a Yamabe-type variational problem associated with the curl operator on middle-degree forms on a closed oriented Riemannian manifold $(M^n,g)$ with $n\equiv3\pmod4$. The corresponding curl--Yamabe constant is defined by minimizing a conformally invariant, gauge-invariant quotient built from $\norm{\curl\alpha}_{\frac{2n}{n+1}}$ and $\int_M\langle\curl\alpha,\alpha\rangle$, and its Euler--Lagrange equation is the critical curl--Yamabe equation $\curl\alpha=\frac{n+1}{2}|\alpha|^{\frac{2}{n-1}}\alpha$. Using sharp curl--Sobolev inequalities on the sphere and an Aubin-type cut-off construction, we prove that $0<Y_{\rm curl}(M,[g])\le Y_{\rm curl}(\S^n)$ for every such manifold, and that the infimum is attained whenever $Y_{\rm curl}(M,[g])<Y_{\rm curl}(\S^n)$. Moreover, for $n>3$ we establish the strict inequality for manifolds that are not locally conformally flat, via a test-form expansion detecting the Weyl tensor, and deduce existence of a minimizer solving the curl--Yamabe equation.
\end{abstract}

\maketitle
\tableofcontents

\section{Introduction}

Let $(M^n,g)$ be a (connected) closed oriented Riemannian manifold with $n\equiv 3\pmod 4$, and let $\alpha$ be an $\frac{n-1}{2}$-form on $M$. Set
\eq{\label{eq:def}
    J(\alpha) \coloneqq J_g(\alpha) \coloneqq \frac{ \Big(\int_{M} \abs{\curl\alpha}^{\frac{2n}{n+1}} \,\rdV_g \Big)^{\frac{n+1}{n}} }{ \int_{M} \< \curl\alpha, \alpha\> \,\rdV_g },
}
Define the curl--Yamabe constant
\eq{\label{eq:inf}
    Y_{{\rm curl}}(M,[g]) :=\inf _{\int _M  \< \curl \alpha, \alpha\>>0}J(\alpha),
}
which is conformally invariant. The functional $J$ is also gauge invariant:
\eq{\label{eq:gauge}
    J(\alpha+\gamma) = J(\alpha)
}
for any closed $\gamma$. After choosing a suitable representative and rescaling, we may assume that a minimizer $\alpha$ (if it exists) solves the Euler--Lagrange equation
\eq{\label{eq:Yamabe}
    \curl\alpha = \frac{n+1}{2}\abs{\alpha}^{\frac{2}{n-1}}\alpha,
}
which is called a curl--Yamabe equation, see Appendix~\ref{appendix_regularity} for details. Analogously to the Yamabe problem, we ask the following Yamabe problem related to the curl operator: 

\medskip

\noindent{\bf Curl-Yamabe problem}. {\it Does every closed oriented manifold of dimension $n\equiv3 \pmod 4$ admit a nontrivial solution to \eqref{eq:Yamabe}?}

\medskip

In the case $(M,g)=(\S^n,g_{{\rm st}})$, our previous work \cite{WZ26curl} proved that
\eq{
    Y_{{\rm curl}}(\S^n) = \frac{n+1}{2}\omega_n^{\frac{1}{n}},
}
where $\omega_n\coloneqq\abs{\S^n}$. Moreover, there is a family of minimizers, which will be used in the paper later.  We now state our main results.

\begin{theorem}\label{thm:main_1}
Let $(M^n,g)$ be a closed oriented Riemannian manifold with $n\equiv 3\pmod 4$. We have
\eq{\label{eq:main_1}
    0<Y_{{\rm curl}}(M,[g]) \leq Y_{{\rm curl}}(\S^n).
}
\end{theorem}

\begin{theorem}\label{thm:main_2}
Let $(M^n,g)$ be a closed oriented Riemannian manifold with $n\equiv 3\pmod 4$. If
\eq{\label{eq:strict}
    Y_{{\rm curl}}(M,[g]) < Y_{{\rm curl}}(\S^n),
}
then the infimum \eqref{eq:inf} is attained by some $\alpha\in C^{0,\tau}(M) \cap W^{1,2}(M)\cap C^\infty(\{\alpha\neq0\})$ for some $\tau>0$.
\end{theorem}

\begin{theorem}\label{thm:main_3}
Let $(M^n,g)$ be a closed oriented Riemannian manifold with $n\equiv 3\pmod 4$ and $n>3$. If $(M^n,g)$ is not locally conformally flat, then
\eq{
    Y_{{\rm curl}}(M,[g]) < Y_{{\rm curl}}(\S^n).
}
In particular, the infimum \eqref{eq:inf} is attained by some $\alpha\in C^{0,\tau}(M) \cap W^{1,2}(M)\cap C^\infty(\{\alpha\neq0\})$ for some $\tau>0$.
\end{theorem}

We remark that in the case $n\equiv1\pmod4$, $n\geq5$, one can instead define the curl operator by $\curl\coloneqq i\ast\rd$, where $i$ is the imaginary unit, in order to preserve self-adjointness. Using the same argument, one can analogously obtain the same results as Theorem~\ref{thm:main_1}, Theorem~\ref{thm:main_2}, and Theorem~\ref{thm:main_3} for complex-valued differential forms.
In this paper, we focus on the case $n\equiv3\pmod4$.

The classical Yamabe problem asks whether, for a closed smooth Riemannian manifold $(M^n,g)$ with $n\ge 3$, there exists a conformal metric $\tilde g=u^{\frac{4}{n-2}}g$ whose scalar curvature is constant; equivalently one seeks a positive smooth solution of the Yamabe equation $-\frac{4(n-1)}{n-2}\Delta_g u+R_g u=\lambda u^{\frac{n+2}{n-2}}$, obtained as an Euler--Lagrange equation for minimizing the normalized total scalar curvature (the Yamabe functional) in the conformal class. Yamabe initiated this variational program \cite{Yamabe60}. Trudinger supplied the missing analytic estimates and corrected the argument \cite{Trudinger68}. A key step is a strict comparison with the round sphere: Aubin proved that if $(M^n,g)$ ($n\geq6$) is not locally conformally flat, then $Y(M,[g])<Y(\S^n)$, which implies the existence of a smooth minimizer that realizes a constant scalar curvature \cite{Aubin76}. The remaining difficulties are the locally conformally flat case and the lower dimensional case. Schoen \cite{Schoen84} then introduced the use of the positive mass theorem to establish the same strict inequality for any $(M,[g])$ that is not conformally equivalent to the round sphere, hence completing the resolution of the Yamabe problem. Beyond existence, the Yamabe constant $Y(M,[g])$ is a fundamental conformal invariant, and has been studied in many later works, see e.g. \cite{LeeParker87}.

There are many Yamabe type problems. Here we just mention one, which is very close to our Yamabe problem, the spinorial Yamabe problem. On a closed spin manifold one studies the conformally covariant Dirac operator $\D_g$ and the conformal quantity
\eq{
  \lambda_{\min}^+(M,[g])\coloneqq \inf_{\tilde g\in[g]}\,\lambda_1^+(\D_{\tilde g})\,\operatorname{Vol}(M,\tilde g)^{1/n},
}
and asks whether this infimum is attained. Very much in the spirit of our quotient formulation, $\lambda_{\min}^+(M,[g])$ can equivalently be written (for any fixed background $g\in[g]$) as an infimum of a Dirac--Sobolev quotient over smooth spinors
\eq{
  \lambda_{\min}^+(M,[g]) = \inf_{\int \< \D\varphi, \varphi\>>0}\frac{ \Big(\int_{M} \abs{\D\varphi}^{\frac{2n}{n+1}} \,\rdV_g \Big)^{\frac{n+1}{n}} }{ \int_{M} \< \D\varphi, \varphi\> \,\rdV_g },
}
whose Euler--Lagrange equation is a critical spinorial Yamabe equation, see e.g. \cites{Ammann03,AGHM08}. Historically, the problem is rooted in the Bochner--Lichnerowicz formula and early eigenvalue estimates for $\D_g$ \cite{Lichnerowicz63}, and was shaped by Hijazi's inequality relating $\lambda_{\min}^+(M,[g])$ to the classical Yamabe constant \cite{Hijazi86}. 
A particularly useful tool for proving the strict inequality $\lambda_{\min}^+(M,[g])<\lambda_{\min}^+(\S^n)$ is the mass endomorphism, defined as the constant term in the expansion of the Green's function of $\D_g$. If this endomorphism has a positive eigenvalue, then one obtains the strict inequality and hence existence of minimizers. This is analogous to the role of the ADM mass in Schoen's approach to the classical Yamabe problem \cite{AHM06}. More recently, Sire--Xu established further cases of the desired strict inequality (and hence existence of minimizers) in higher dimensions. Related blow-up and compactness analysis was developed by Isobe--Sire--Xu \cites{SireXu21,ISX24}. Nevertheless, it remains open in general whether every conformal spin manifold not conformally equivalent to the round sphere satisfies the strict inequality.

\smallskip

\noindent\textbf{Organization of the paper.} In Section~\ref{sec:prelim} we recall basic exterior-calculus identities and the conformal covariance of the curl operator on $\Omega^{\frac{n-1}{2}}$, and we record the conformal and gauge invariances of the functional defining $Y_{\rm curl}(M,[g])$. In Section~3 we prove Theorem~\ref{thm:main_1}: positivity and the sharp upper bound by the sphere, using the sharp curl--Sobolev inequality on $\S^n$ together with an Aubin-type cut-off argument based on Euclidean extremals. In Section~4 we establish Theorem~\ref{thm:main_2} by a concentration--compactness analysis, showing that the infimum is attained under the strict inequality $Y_{\rm curl}(M,[g])<Y_{\rm curl}(\S^n)$. In Section~5 we prove Theorem~\ref{thm:main_3} by constructing refined test forms whose asymptotic expansion detects the Weyl tensor, yielding the strict inequality (hence existence of minimizers) on manifolds that are not locally conformally flat. 
Appendix~\ref{appendix_regularity} completes the regularity argument.

\section{Preliminaries}\label{sec:prelim}

Throughout, we assume that $n\equiv 3\pmod 4$, so that the curl operator is self-adjoint.
\subsection{Exterior calculus}

Let $(M,g)$ be an oriented Riemannian manifold of dimension $n\equiv 3\pmod 4$. For each $0\le p\le n$, we denote by $\Omega^p(M)$ the space of $p$-forms and equip it with the inner product induced by $g$. The exterior derivative is denoted by $\rd:\Omega^p\to\Omega^{p+1}$, with the convention $\Omega^{n+1}=0$. We fix the orientation determined by the volume form $\rdV$. The Hodge star operator $*:\Omega^p\to\Omega^{n-p}$ is characterized by
\eq{\label{eq:def_Hodge_star}
    \alpha\w*\beta = \<\alpha, \beta\>\,\rdV.
}
The codifferential, i.e. the $L^2$-adjoint of $\rd$, is
\eq{
    \rd^*\coloneqq (-1)^{n(p+1)+1}*\rd*:\Omega^p\ra\Omega^{p-1}.
}

Since $n\equiv 3\pmod 4$, one has $\rd^*=-*\rd*$ on $\Omega^{\frac{n-1}{2}}$. We recall the standard identities
\eq{
    *^2 = (-1)^{p(n-p)}\,{\rm id},\qquad \rd^2 = (\rd^*)^2=0.
}

The Hodge Laplacian is defined by ${\Delta_H \coloneqq \rd\rd^* + \rd^*\rd}$
and the rough Laplacian ${\nabla^*\nabla \coloneqq -{\rm tr}\nabla^2}$.
These operators are related by the Weitzenb\"ock formula
\eq{
    \Delta_H = \nabla^*\nabla + \sum_{i,j} e^j \w e_i \ip R(e_i, e_j),
}
where $\{e_i\}$ is a local orthonormal frame, $\{e^i\}$ is the dual $1$-form coframe, $\ip$ is the interior multiplication, and $R$ is the curvature tensor. On the standard round sphere $\S^n$, acting on $p$-forms, this reduces to
\eq{\label{sphere_Weitzenbock}
    \Delta_H = \nabla^*\nabla + p(n-p).
}

We record several identities that will be used repeatedly.

\begin{proposition}\label{basic_form}
Let $(M^n,g)$ be a closed oriented Riemannian manifold of dimension $n\equiv 3\pmod 4$. For any $\alpha\in\Omega^p(M)$, $\beta\in\Omega^{n-p-1}(M)$ and $X\in\Gamma(TM)$, we have
\begin{enumerate}
    \item $X\ip*\!\alpha=(-1)^p*(X^\flat\w\alpha)$;
    \item $\<X\ip*\!\alpha,\beta\>=(-1)^p\<\alpha, X\ip*\beta\>$;
    \item $e^i\w e_j\ip\alpha + e_j\ip e^i\w\alpha = \delta_{ij}\alpha$;
    \item $\mathcal{L}_X\alpha = \rd (X\ip\alpha) + X\ip\rd\alpha$.
\end{enumerate}
\end{proposition}

\subsection{The curl operator}\label{sec2.2}

Let $n\equiv 3\pmod 4$. We define the curl operator by
\eq{\label{eq:curl-def}
    \curl \coloneqq *\rd: \Omega^{\frac{n-1}{2}}\ra \Omega^{\frac{n-1}{2}}.
}
When $n=3$ and $\alpha$ is the $1$-form dual to a vector field $X$, definition \eqref{eq:curl-def} coincides with the usual vector-calculus curl.

A basic computation using $\<\alpha,\beta\>\rdV=\alpha\wedge*\beta$ and Stokes' theorem shows that on a closed manifold $M$ the curl operator is formally self-adjoint on $\Omega^{\frac{n-1}{2}}(M)$:
\eq{\label{eq:curl-selfadj}
    \int_M \< \curl\alpha,\beta\>\,\rdV
    = \int_M \< \alpha,\curl\beta\>\,\rdV.
}
In particular,
\eq{\label{eq:curl-energy}
    \int_M \< \curl\alpha,\alpha\>\,\rdV = \int_M \alpha\wedge \rd\alpha.
}

Let $\tilde g=\sigma^2g$ for some smooth function $\sigma>0$. The standard conformal covariance is
\eq{\label{conformal_change}
    *_{\tilde{g}} = \sigma^{n-2p} *_{g}, \quad \<\cdot,\cdot\>_{\tilde{g}} = \sigma^{-2p}\<\cdot,\cdot\>_g \quad\hbox{on}\ \Omega^p.
}
In particular, for any $\alpha\in\Omega^{\frac{n-1}{2}}$,
\eq{
    \curl_{\tilde{g}}\alpha = *_{\tilde{g}}\rd\alpha = \sigma^{-1}\,\curl_g\alpha.
}
Together with \eqref{conformal_change} this yields
\eq{
    \abs{\curl_{\tilde{g}}\alpha}_{\tilde{g}}^{\frac{2n}{n+1}} = \sigma^{-n}\abs{\curl_g\alpha}_g^{\frac{2n}{n+1}},\qquad
    \<\curl_{\tilde{g}}\alpha, \alpha\>_{\tilde{g}} = \sigma^{-n}\<\curl_g\alpha,\alpha\>_g.
}
Consequently, the following two quantities
\eq{\label{functionals}
    \int_M \abs{\curl\alpha}^{\frac{2n}{n+1}}\,\rdV, \qquad \int_M \<\curl\alpha,\alpha\>\,\rdV
}
are conformally invariant, and hence so is $J(\alpha)$.

\section{Proof of Theorem~\ref{thm:main_1}}

In this section, we use an Aubin-type cut-off argument to prove Theorem~\ref{thm:main_1}. Using this method, the analogous result was proved by Aubin \cite{Aubin76} for the classical Yamabe problem (see also \cite{LeeParker87} for a nice survey), by Ammann--Grosjean--Humbert--Morel \cite{AGHM08} for the spinorial Yamabe problem, and by Djadli--Hebey--Ledoux \cite{DHL2000} for the higher-order Yamabe problem.

\subsection{Positivity}
It is clear that $Y_{\rm curl}(M,[g])\ge 0$.
We first prove that
\eq{
    Y_{{\rm curl}}(M,[g]) > 0.
}

We need Gaffney's inequality for any differential form $\alpha\in W^{1,p}$ with $1<p<\infty$ (see e.g. \cite[Proposition 4.10]{Scott95})
\eq{\label{eq:Gaffney}
    \norm{\alpha}_{W^{1,p}} \lesssim \norm{\rd\alpha}_{L^p} + \norm{\rd^*\alpha}_{L^p} + \norm{\alpha}_{L^p}.
}
In the paper we use the notation $f\lesssim{g}$ for $f\le Cg$ with a  constant $C$ independent of $f$ and $g$.
Recall the Hodge decomposition
\eq{
    \Omega^{\frac{n-1}{2}}(M) = \rd\Omega^{\frac{n-3}{2}}(M) \oplus \rd^*\Omega^{\frac{n+1}{2}}(M) \oplus \mathcal{H}^{\frac{n-1}{2}}(M),
}
where the components are the spaces of exact, co-exact, and harmonic $\frac{n-1}{2}$-forms, respectively. In view of the gauge invariance \eqref{eq:gauge}, without loss of generality, we may assume
\eq{
    \rd^*\alpha=0, \qquad \alpha\perp\mathcal{H}^{\frac{n-1}{2}}(M).
}
Gaffney's inequality \eqref{eq:Gaffney} then gives
\eq{\label{eq:Gaffney_1}
    \norm{\alpha}_{W^{1,\frac{2n}{n+1}}} \lesssim \norm{\rd\alpha}_{L^\frac{2n}{n+1}} + \norm{\alpha}_{L^\frac{2n}{n+1}}.
}
We now claim that
\eq{\label{eq:Gaffney_2}
    \norm{\alpha}_{W^{1,\frac{2n}{n+1}}} \lesssim \norm{\rd\alpha}_{L^\frac{2n}{n+1}}.
}
If this is not true, then there exists $\{\alpha_k\}_{k=1}^\infty$ with $\rd^*\alpha_k=0$ and $\alpha_k\perp\mathcal{H}^{\frac{n-1}{2}}(M)$, such that
\eq{\label{eq:Gaffney_3}
    \norm{\alpha_k}_{L^\frac{2n}{n+1}}=1, \qquad \norm{\rd\alpha_k}_{L^\frac{2n}{n+1}}\to0.
}
Together with \eqref{eq:Gaffney_1}, we see that after passing to a subsequence, $\alpha_k\to\alpha$ strongly in $L^{\frac{2n}{n+1}}$. Then $\rd^*\alpha=0$, $\alpha\perp\mathcal{H}^{\frac{n-1}{2}}(M)$, and \eqref{eq:Gaffney_3} implies that $\rd\alpha=0$. Hence $\alpha=0$, which contradicts \eqref{eq:Gaffney_3}. The claim \eqref{eq:Gaffney_2} then follows.

Using the critical Sobolev embedding and H\"older's inequality, we have
\eq{
    \Abs{ \int_M \alpha\w\rd\alpha } \leq \norm{\alpha}_{L^\frac{2n}{n-1}}\norm{\rd\alpha}_{L^\frac{2n}{n+1}} \lesssim \norm{\rd\alpha}_{L^\frac{2n}{n+1}}^2.
}
Hence the positivity follows.

\subsection{The Euclidean minimizers}

Let $\delta$ be the standard Euclidean metric on $\R^n$. By conformal invariance, we see that
\eq{
    Y_{{\rm curl}}(\R^n,[\delta]) = Y_{{\rm curl}}(\S^n) = \frac{n+1}{2}\omega_n^{\frac{1}{n}}.
}
Moreover, \cite[Theorem 1.1]{WZ26curl} implies that the minimizers are precisely the image of positive Killing $\frac{n-1}{2}$-forms from $\S^n$ to $\R^n$ via the inverse stereographic projection, modulo $\ker(\rd)$ and ${\rm Conf}^+(\S^n)$.

Let $v_0\in\Omega^{\frac{n-1}{2}}(\R^n)$ be such a Euclidean minimizer, namely
\eq{\label{eq:0.1}
    v_0 = \Psi^*\alpha_0,
}
where $\alpha_0$ is a positive Killing $\frac{n-1}{2}$-form on $\S^n$ with unit length, and $\Psi$ is the inverse stereographic projection given by
\eq{\label{eq:0.1.1}
    \Psi:\R^n\to\S^n\backslash\{N\}, \qquad \Psi(y) = \Big( \frac{2y}{1+\abs{y}^2}, \frac{\abs{y}^2-1}{1+\abs{y}^2} \Big).
}

\begin{lemma}
We have
\eq{
    \abs{v_0(y)}_\delta = O(\abs{y}^{-(n-1)}), \qquad \abs{\rd v_0(y)}_\delta = O(\abs{y}^{-(n+1)}).
}
\end{lemma}
\begin{proof}
It is clear that
\eq{\label{eq:0.2}
    \Psi^*g_{{\rm st}} = \Big( \frac{2}{1+\abs{y}^2} \Big)^2 \delta.
}
Since
\eq{
    \abs{v_0}_{\Psi^*g_{{\rm st}}} = \abs{\alpha_0}_{g_{{\rm st}}} = 1,
}
using \eqref{conformal_change} we have
\eq{\label{eq:0.2.1}
    \abs{v_0(y)}_\delta = \Big( \frac{2}{1+\abs{y}^2} \Big)^{\frac{n-1}{2}}.
}
Hence the first claim follows. Next, using the positive Killing form equation
\eq{\label{eq:0.3}
    \curl\alpha_0 = \frac{n+1}{2}\alpha_0,
}
we see that
\eq{
    \abs{\rd\alpha_0}_{g_{{\rm st}}} = \abs{\curl\alpha_0}_{g_{{\rm st}}} = \frac{n+1}{2}.
}
Again \eqref{conformal_change} implies that
\eq{
    \abs{\rd v_0(y)}_\delta = \frac{n+1}{2}\Big( \frac{2}{1+\abs{y}^2} \Big)^{\frac{n+1}{2}}.
}
Hence the second claim follows.
\end{proof}

Given a small $\epsilon>0$, set
\eq{\label{eq:0.9}
    v_\epsilon(y) \coloneqq \epsilon^{-\frac{n-1}{2}}v_0(y/\epsilon).
}
Then
\eq{\label{eq:1}
    \abs{v_\epsilon(y)}_\delta \leq C\epsilon^{\frac{n-1}{2}}\abs{y}^{-(n-1)}, \qquad \abs{\rd v_\epsilon(y)}_\delta \leq C\epsilon^{\frac{n+1}{2}}\abs{y}^{-(n+1)}.
}
One can easily check that
\eq{\label{eq:1.1}
    \int_{\R^n} \abs{\rd v_\epsilon(y)}^{\frac{2n}{n+1}}_\delta \rd y = \int_{\R^n} \abs{\rd v_0(y)}^{\frac{2n}{n+1}}_\delta \rd y,
}
and
\eq{\label{eq:1.2}
    \int_{\R^n} v_\epsilon\w\rd v_\epsilon = \int_{\R^n} v_0\w\rd v_0.
}

\subsection{The cut-off function}\label{sec3.3}

Fix $x_0\in M$. Choosing normal coordinates around $x_0$, we may identify the geodesic ball $B_r^M(x_0)$ with the Euclidean ball $B_r\coloneqq B_r(0)\subset\R^n$, for a fixed small $r>0$. In this small geodesic ball, we have
\eq{\label{eq:1.3}
    g = \delta + O(\abs{y}^2),
}
and hence
\eq{\label{eq:1.4}
    \rdV_g = \big( 1 + O(\abs{y}^2) \big) \rd y.
}

Choose a cut-off function $\eta\in C^\infty(\R^n)$ such that
\eq{\label{eq:1.5}
    \eta\equiv1 \quad \hbox{in}\ \ B_{r/2}, \qquad \eta\equiv0 \quad \hbox{in}\ \ \R^n\backslash B_r, \qquad \abs{\rd\eta}\leq \frac{4}{r}.
}
Set
\eq{\label{eq:1.6}
    \beta_\epsilon \coloneqq \eta v_\epsilon.
}
By a slight abuse of notation, we regard $\beta_\epsilon$ as an $\frac{n-1}{2}$-form supported on $B_r^M(x_0)\subset M$ or on $B_r\subset \R^n$, according to the context.
On the one hand,
\eq{
    \beta_\epsilon\w\rd\beta_\epsilon = \eta v_\epsilon\w\rd(\eta v_\epsilon) = \eta v_\epsilon\w(\eta\,\rd v_\epsilon + \rd\eta\w v_\epsilon) = \eta^2v_\epsilon\w\rd v_\epsilon,
}
where in the last step we used $v_\epsilon\w v_\epsilon=0$ since $\frac{n-1}{2}$ is odd. Hence
\eq{\label{eq:2}
    \int_M \beta_\epsilon\w\rd\beta_\epsilon = \int_{B_r} \eta^2 v_\epsilon\w\rd v_\epsilon = \int_{\R^n} v_0\w\rd v_0 - \int_{\R^n}(1-\eta^2)v_\epsilon\w\rd v_\epsilon.
}
Note that $1-\eta^2$ is supported on $\R^n\backslash B_{r/2}$, on which we have
\eq{
    \abs{v_\epsilon\w\rd v_\epsilon}_\delta (y) \lesssim \epsilon^n\abs{y}^{-2n} = O(\epsilon^n),
}
where we used \eqref{eq:1}. Therefore, by \eqref{eq:2} we have
\eq{\label{eq:2.1}
    \int_M \beta_\epsilon\w\rd\beta_\epsilon = \int_{\R^n} v_0\w\rd v_0 + O(\epsilon^n) > 0.
}
On the other hand,
\eq{\label{eq:3}
    \abs{\rd\beta_\epsilon}_\delta^{\frac{2n}{n+1}} = \abs{\eta\,\rd v_\epsilon+\rd\eta\w v_\epsilon}_\delta^{\frac{2n}{n+1}} \leq \abs{\eta\,\rd v_\epsilon}_\delta^{\frac{2n}{n+1}} + C\Big( \abs{\eta\,\rd v_\epsilon}_\delta^{\frac{n-1}{n+1}}\abs{\rd\eta\w v_\epsilon}_\delta + \abs{\rd\eta\w v_\epsilon}_\delta^{\frac{2n}{n+1}} \Big).
}
Note that $\rd\eta$ is supported on $B_r\backslash B_{r/2}$, on which we have
\eq{\label{eq:4}
    \abs{\eta\,\rd v_\epsilon}_\delta \lesssim \epsilon^{\frac{n+1}{2}}\abs{y}^{-(n+1)} = O(\epsilon^{\frac{n+1}{2}}),
}
and
\eq{\label{eq:5}
    \abs{\rd\eta\w v_\epsilon}_\delta \lesssim \frac{4}{r}\,\epsilon^{\frac{n-1}{2}}\abs{y}^{-(n-1)} = O(\epsilon^{\frac{n-1}{2}}).
}
Combining \eqref{eq:3}, \eqref{eq:4}, and \eqref{eq:5} gives
\eq{\label{eq:6}
    \int_{\R^n} \abs{\rd\beta_\epsilon}_\delta^{\frac{2n}{n+1}} \rd y &\leq \int_{\R^n} \abs{\eta\,\rd v_\epsilon}_\delta^{\frac{2n}{n+1}} \rd y + C\int_{\R^n} \abs{\eta\,\rd v_\epsilon}_\delta^{\frac{n-1}{n+1}}\abs{\rd\eta\w v_\epsilon} \,\rd y + C\int_{\R^n} \abs{\rd\eta\w v_\epsilon}_\delta^{\frac{2n}{n+1}} \rd y \\
    &\leq \int_{\R^n} \abs{\rd v_\epsilon}_\delta^{\frac{2n}{n+1}} \rd y + O(\epsilon^{n-1}) + O(\epsilon^{\frac{n(n-1)}{n+1}}) \\
    &= \int_{\R^n} \abs{\rd v_0}_\delta^{\frac{2n}{n+1}} \rd y + O(\epsilon^{\frac{n(n-1)}{n+1}}),
}
where we used \eqref{eq:1.1} in the last step. In order to relate the integral on $M$ and that on $\R^n$, we split the integral into two parts, inside and outside $B_{\sqrt{\epsilon}}$. Inside the small ball, using \eqref{eq:1.3} and \eqref{eq:1.4} we have
\eq{\label{eq:6.1}
    \abs{\rd\beta_\epsilon(y)}_g^{\frac{2n}{n+1}}\rdV_g = (1+O(\abs{y}^2))\abs{\rd\beta_\epsilon(y)}_\delta^{\frac{2n}{n+1}}\rd y,
}
hence
\eq{\label{eq:7}
    \int_{B_{\sqrt{\epsilon}}} \abs{\rd \beta_\epsilon(y)}_\delta^{\frac{2n}{n+1}} \rd y = \int_{B_{\sqrt{\epsilon}}^M(x_0)} \abs{\rd \beta_\epsilon(y)}_g^{\frac{2n}{n+1}} \rdV_g + O(\epsilon).
}
Outside the small ball, using \eqref{eq:1} we have
\eq{\label{eq:8}
    \int_{\R^n\backslash B_{\sqrt{\epsilon}}} \abs{\rd v_\epsilon(y)}_\delta^{\frac{2n}{n+1}} \rd y \leq C\epsilon^n\int_{\sqrt{\epsilon}}^\infty s^{-2n}\cdot s^{n-1} \rd s = O(\epsilon^{\frac{n}{2}}),
}
and
\eq{\label{eq:9}
    \int_{\R^n\backslash B_{\sqrt{\epsilon}}} \abs{\rd\eta\w v_\epsilon(y)}_\delta^{\frac{2n}{n+1}} \rd y \leq C\epsilon^{\frac{n(n-1)}{n+1}} \rVol(B_r\backslash B_{r/2}) = O(\epsilon^{\frac{n(n-1)}{n+1}}).
}
We may assume $\sqrt{\epsilon}<r$. Using \eqref{eq:6.1} we have
\eq{\label{eq:10}
    \int_M \abs{\rd\beta_\epsilon(y)}_g^{\frac{2n}{n+1}}\rdV_g &\leq \int_{B_r} \abs{\rd\beta_\epsilon(y)}_\delta^{\frac{2n}{n+1}} \rd y + C_1\int_{B_r} \abs{y}^2\abs{\rd\beta_\epsilon(y)}_\delta^{\frac{2n}{n+1}} \rd y \\
    &\leq \int_{\R^n} \abs{\rd\beta_\epsilon(y)}_\delta^{\frac{2n}{n+1}} \rd y + C_1\epsilon\int_{B_{\sqrt{\epsilon}}} \abs{\rd\beta_\epsilon(y)}_\delta^{\frac{2n}{n+1}} \rd y + C_1r^2\int_{B_r\backslash B_{\sqrt{\epsilon}}} \abs{\rd\beta_\epsilon(y)}_\delta^{\frac{2n}{n+1}} \rd y.
}
Combining \eqref{eq:6}, \eqref{eq:7}, \eqref{eq:8}, \eqref{eq:9}, and \eqref{eq:10} gives
\eq{\label{eq:11}
    \int_M \abs{\rd\beta_\epsilon(y)}_g^{\frac{2n}{n+1}}\rdV_g \leq \int_{\R^n} \abs{\rd v_0}_\delta^{\frac{2n}{n+1}} \rd y + o_\epsilon(1).
}

We now prove Theorem~\ref{thm:main_1}.

\begin{proof}[Proof of Theorem~\ref{thm:main_1}]
By \eqref{eq:2.1} and \eqref{eq:11} we have
\eq{
    Y_{{\rm curl}}(M,[g]) \leq J_g(\beta_\epsilon) \leq \frac{ \Big(\int_{\R^n} \abs{\rd v_0}_\delta^{\frac{2n}{n+1}} \rd y + o_\epsilon(1) \Big)^{\frac{n+1}{n}} }{ \int_{\R^n} v_0\w\rd v_0 + O(\epsilon^n) } \to Y_{{\rm curl}}(\S^n) = \frac{n+1}{2}\omega_n^{\frac{1}{n}} \quad\hbox{as}\ \ \epsilon\to0.
}
Hence we complete the proof.
\end{proof}

\section{Proof of Theorem~\ref{thm:main_2}}\label{sec4}

The idea is to use the concentration--compactness argument, but with a particular modification in order to deal with the half-ellipticity of the curl operator.

Let $\{\alpha_k\}_{k=1}^\infty$ be a minimizing sequence such that
\eq{\label{eq:normalization}
    \int_M \alpha_k\w\rd\alpha_k = 1, \qquad \Big( \int_M \abs{\rd\alpha_k}^{\frac{2n}{n+1}} \rdV_g \Big)^{\frac{n+1}{n}} \to Y_{{\rm curl}}(M,[g]).
}
Again, we may assume for each $k\geq1$
\eq{
    \rd^*\alpha_k=0, \qquad \alpha_k\perp\mathcal{H}^{\frac{n-1}{2}}(M),
}
hence \eqref{eq:Gaffney_2} implies that $\{\alpha_k\}$ is bounded in $W^{1,\frac{2n}{n+1}}$. After passing to a subsequence, we have $\alpha_k \rightharpoonup \alpha$ in $W^{1,\frac{2n}{n+1}}$. Set $\gamma_k\coloneqq\alpha_k-\alpha$, then $\gamma_k \rightharpoonup 0$ in $W^{1,\frac{2n}{n+1}}$.

\begin{lemma}
After passing to a subsequence, we have
\eq{
    \abs{\rd\alpha_k}^{\frac{2n}{n+1}}\rdV_g \overset{\ast}{\rightharpoonup} \theta, \qquad \abs{\rd\gamma_k}^{\frac{2n}{n+1}}\rdV_g \overset{*}{\rightharpoonup} \mu, \qquad \gamma_k\w\rd\gamma_k \overset{\ast}{\rightharpoonup} \nu. 
}
\end{lemma}
\begin{proof}
It suffices to show that the three sequences of measures are all bounded. The first claim follows from \eqref{eq:normalization}. Using the weak lower semi-continuity we have
\eq{
    \norm{\rd\alpha}_{L^\frac{2n}{n+1}} \leq \liminf_{k\to\infty}\, \norm{\rd\alpha_k}_{L^\frac{2n}{n+1}} < \infty.
}
By Minkowski's inequality,
\eq{\label{eq:12}
    \sup_k\,\norm{\rd\gamma_k}_{L^\frac{2n}{n+1}} \leq \sup_k\,\norm{\rd\alpha_k}_{L^\frac{2n}{n+1}} + \norm{\rd\alpha}_{L^\frac{2n}{n+1}} < \infty.
}
Hence the second claim then follows.

Finally, by \eqref{eq:Gaffney_2} we see that $\{\alpha_k\}$ is bounded in $W^{1,\frac{2n}{n+1}}$, hence
\eq{
    \norm{\alpha}_{W^{1,\frac{2n}{n+1}}} \leq \liminf_{k\to\infty}\, \norm{\alpha_k}_{W^{1,\frac{2n}{n+1}}} < \infty.
}
By Minkowski's inequality,
\eq{\label{eq:13}
    \sup_k\,\norm{\gamma_k}_{L^{\frac{2n}{n-1}}} \lesssim \sup_k\,\norm{\gamma_k}_{W^{1,\frac{2n}{n+1}}} \leq \sup_k\,\norm{\alpha_k}_{W^{1,\frac{2n}{n+1}}} + \norm{\alpha}_{W^{1,\frac{2n}{n+1}}} < \infty.
}
The third claim follows from \eqref{eq:12}, \eqref{eq:13}, and H\"older's inequality.
\end{proof}

It is clear that $\mu$ is a non-negative finite measure, while $\nu$ could be a signed finite measure. The Hahn--Jordan Decomposition Theorem implies that
\eq{
    \nu = \nu^+ - \nu^- \qquad\hbox{with}\quad \nu^+\geq0, \quad \nu^-\geq0, \quad \nu^+\perp\nu^-.
}
In other words, there exists a decomposition (unique up to a $\abs{\nu}$-negligible set)
\eq{
    M = P \cup N \qquad\hbox{with}\quad P\cap N=\emptyset,
}
such that for any measurable subset $E$,
\eq{
    E\subset P \implies \nu(E)\geq0, \qquad E\subset N \implies \nu(E)\leq0.
}

\begin{lemma}\label{lem4.2}
$\nu^+$ is supported on at most countably many atoms $\{x_j\}_{j=1}^\infty$. Moreover, if we denote
\eq{
    \mu_j \coloneqq \mu(\{x_j\}), \qquad \nu_j \coloneqq \nu^+(\{x_j\}),
}
then
\eq{
    \frac{n+1}{2}\omega_n^{\frac{1}{n}}\nu_j \leq \mu_j^{\frac{n+1}{n}}.
}
\end{lemma}
\begin{proof}
Given any small geodesic ball $B_\rho$ and any cut-off function $\eta\in C^\infty_c(B_\rho)$, the sharp curl--Sobolev inequality gives that
\eq{\label{eq:14}
    \frac{n+1}{2}\omega_n^{\frac{1}{n}} \int_{\R^n} \eta\gamma_k\w\rd(\eta\gamma_k) \leq \Big( \int_{\R^n} \abs{\rd(\eta\gamma_k)}_\delta^{\frac{2n}{n+1}} \rd y \Big)^{\frac{n+1}{n}}.
}
Since $\frac{n-1}{2}$ is odd, we have
\eq{\label{eq:15}
    \eta\gamma_k\w\rd(\eta\gamma_k) = \eta^2\gamma_k\w\rd\gamma_k.
}
Minkowski's inequality gives that
\eq{\label{eq:16}
    \norm{\rd(\eta\gamma_k)}_{L^{\frac{2n}{n+1}}} = \norm{\eta\rd\gamma_k+\rd\eta\w\gamma_k}_{L^{\frac{2n}{n+1}}} \leq \norm{\eta\rd\gamma_k}_{L^{\frac{2n}{n+1}}} + \norm{\rd\eta\w\gamma_k}_{L^{\frac{2n}{n+1}}}.
}
Recall \eqref{eq:1.3} that
\eq{\label{eq:17}
    g = \delta + o_\rho(1) \qquad\hbox{in}\ \ B_\rho.
}
Since $\gamma_k \rightharpoonup 0$ in $W^{1,\frac{2n}{n+1}}$, we have $\gamma_k\to0$ strongly in $L^{\frac{2n}{n+1}}$. Now insert \eqref{eq:15} and \eqref{eq:16} into \eqref{eq:14}, and let $k\to\infty$. Using \eqref{eq:17} we have
\eq{\label{eq:18}
    \frac{n+1}{2}\omega_n^{\frac{1}{n}} \int_M \eta^2\rd\nu \leq \big( 1 + o_\rho(1) \big) \Big( \int_M \abs{\eta}^{\frac{2n}{n+1}} \rd\mu \Big)^{\frac{n+1}{n}}.
}

Given any compact $K\subset\subset P\cap B_{\rho/2}$, we choose an open set $U$ such that $K\subset U\subset\subset B_\rho$, and choose the cut-off function $\eta$ such that
\eq{
    0\leq\eta\leq1, \qquad \eta|_K\equiv1, \qquad {\rm supp}(\eta)\subset U.
}
It follows that $\nu^-(K)=0$, and hence
\eq{\label{eq:19}
    \int_M \eta^2 \rd\nu = \nu(K) + \int_{U\backslash K} \eta^2 \rd\nu \geq \nu^+(K) - \nu^-(U). 
}
Combine \eqref{eq:18} and \eqref{eq:19}, and then let $U\searrow K$. It follows that
\eq{
    \frac{n+1}{2}\omega_n^{\frac{1}{n}} \nu^+(K) \leq \big( 1 + o_\rho(1) \big) \mu(K)^{\frac{n+1}{n}}, \qquad \forall\, K\subset\subset P\cap B_{\rho/2}.
}
It further implies that
\eq{\label{eq:20}
    \frac{n+1}{2}\omega_n^{\frac{1}{n}} \nu^+(E) \leq \big( 1 + o_\rho(1) \big) \mu(E)^{\frac{n+1}{n}}, \qquad \forall\, \hbox{Borel set } E\subset P\cap B_{\rho/2}.
}

Next, we denote by $\mathcal{A}$ the set of atoms of $\mu$, namely
\eq{
    \mathcal{A} \coloneqq \{ x\in M \mid \mu(\{x\})>0 \}, \qquad\hbox{and}\quad \mathcal{B}\coloneqq M\backslash\mathcal{A}.
}
Since $\mu$ is a finite measure, we see that $\mathcal{A}$ is at most countable, and $\mu$ has no atoms on $\mathcal{B}$. Using the Lindel\"of property of the Riemannian manifold $M$, we see that for any $\tau>0$, there exist at most countably many disjoint Borel sets $\{A_l\}$, such that
\eq{
    \mathcal{B} = \bigcup_l A_l, \qquad\hbox{with}\quad \mu(A_l)\leq\tau,
}
and each $A_l$ is contained in a geodesic ball with radius $\rho/2$. Applying \eqref{eq:20} to $E=A_l\cap P$ and summing over $l$, we have
\eq{\label{eq:21}
    \frac{n+1}{2}\omega_n^{\frac{1}{n}}\nu^+(\mathcal{B}) \leq 2\sum_l \mu(A_l\cap P)^{\frac{n+1}{n}} \leq 2\tau^{\frac{1}{n}}\sum_l \mu(A_l\cap P) \leq 2\tau^{\frac{1}{n}}\mu(M),
}
where we used the non-negativity of $\mu$ in the last step. Letting $\tau\to0$ in \eqref{eq:21} implies that $\nu^+(\mathcal{B})=0$. Hence $\nu^+$ is supported on $\mathcal{A}$, which is at most countable.

Finally, we apply \eqref{eq:18} to $\eta_\rho$ supported on $B_\rho$ such that ${\rm supp}(\eta_\rho)\searrow\{x_j\}$ as $\rho\to0$. It follows that
\eq{
    \frac{n+1}{2}\omega_n^{\frac{1}{n}} \nu_j \leq \mu_j^{\frac{n+1}{n}}.
}
Hence we complete the proof.
\end{proof}

We now prove Theorem~\ref{thm:main_2}.

\begin{proof}[Proof of Theorem~\ref{thm:main_2}]
Assume
\eq{\label{eq:assumption}
    Y_{{\rm curl}}(M,[g]) < Y_{{\rm curl}}(\S^n)=\frac{n+1}{2}\omega_n^{\frac{1}{n}}.
}
Recall \eqref{eq:normalization} that
\eq{\label{eq:22}
    1 &= \int_M \alpha_k\w\rd\alpha_k = \int_M (\alpha+\gamma_k)\w\rd(\alpha+\gamma_k) \\
    &= \int_M \alpha\w\rd\alpha + \int_M \gamma_k\w\rd\gamma_k + \int_M \alpha\w\rd\gamma_k + \int_M \gamma_k\w\rd\alpha.
}
Since $\gamma_k\rightharpoonup0$ in $W^{1,\frac{2n}{n+1}}$, we have (1) $\gamma_k\rightharpoonup0$ in $L^{\frac{2n}{n-1}}$, and (2) $\rd\gamma_k\rightharpoonup0$ in $L^{\frac{2n}{n+1}}$. Since $\alpha\in W^{1,\frac{2n}{n+1}}$, we have (3) $\alpha\in L^{\frac{2n}{n-1}}$, and (4) $\rd\alpha\in L^{\frac{2n}{n+1}}$. By H\"older's inequality, (2) and (3) imply that the third integral goes to $0$; (1) and (4) imply that the fourth integral goes to $0$. Hence letting $k\to\infty$ in \eqref{eq:22} yields
\eq{\label{eq:22.1}
    \int_M \alpha\w\rd\alpha + \nu(M) = 1.
}

We now claim that
\eq{\label{eq:23}
    \theta(M) \geq \int_M \abs{\rd\alpha}^{\frac{2n}{n+1}} \rdV_g + \sum_j \mu_j.
}
In fact, let $B_\rho$ be a small ball that shrinks to an atom $\{x_j\}$ as $\rho\to0$, then Minkowski's inequality implies that
\eq{
    \Abs{ \theta(B_\rho)^{\frac{n+1}{2n}} - \mu(B_\rho)^{\frac{n+1}{2n}} } \leq \norm{\rd\alpha}_{L^{\frac{2n}{n+1}}(B_\rho)}.
}
Letting $\rho\to0$ yields that $\theta(\{x_j\})=\mu(\{x_j\})=\mu_j$. Hence the claim \eqref{eq:23} follows from the lower semi-continuity. The normalization \eqref{eq:normalization} then implies that
\eq{\label{eq:24}
    Y_{{\rm curl}}(M,[g]) = \theta(M)^{\frac{n+1}{n}} \geq \Big( \int_M \abs{\rd\alpha}^{\frac{2n}{n+1}} \rdV_g \Big)^{\frac{n+1}{n}} + \Big( \sum_j \mu_j \Big)^{\frac{n+1}{n}}.
}

Finally, we show that \eqref{eq:22.1} and \eqref{eq:24} complete the proof. Using Lemma~\ref{lem4.2} we have
\eq{\label{eq:24.5}
    \Big( \sum_j \mu_j \Big)^{\frac{n+1}{n}} \geq \sum_j \mu_j^{\frac{n+1}{n}} \geq \frac{n+1}{2}\omega_n^{\frac{1}{n}} \nu^+(M) \geq \frac{n+1}{2}\omega_n^{\frac{1}{n}} \nu(M).
}
Together with assumption \eqref{eq:assumption}, we see that $\nu(M)<1$, and hence by \eqref{eq:22.1} we have $\int_M \alpha\w\rd\alpha>0$. Then by definition of $Y_{{\rm curl}}(M,[g])$, we have
\eq{\label{eq:25}
    \Big( \int_M \abs{\rd\alpha}^{\frac{2n}{n+1}} \rdV_g \Big)^{\frac{n+1}{n}} \geq Y_{{\rm curl}}(M,[g]) \int_M \alpha\w\rd\alpha.
}
Combining with \eqref{eq:24}, we see that $\int_M \alpha\w\rd\alpha\leq1$, hence by \eqref{eq:22.1} we have $\nu(M)\geq0$. Now combining \eqref{eq:assumption}, \eqref{eq:22.1}, \eqref{eq:24}, \eqref{eq:24.5}, and \eqref{eq:25} gives
\eq{
    Y_{{\rm curl}}(M,[g]) \geq Y_{{\rm curl}}(M,[g]) \int_M \alpha\w\rd\alpha + \frac{n+1}{2}\omega_n^{\frac{1}{n}} \nu(M) \geq Y_{{\rm curl}}(M,[g]).
}
Hence all equalities hold. Equality in \eqref{eq:25} implies that $\alpha\in W^{1,\frac{2n}{n+1}}(M)$ is a minimizer.

Finally, using the same argument as in our previous work \cite[Subsection 8.3; Appendix C]{WZ26curl}, one obtains the regularity. We leave the details in Appendix~\ref{appendix_regularity}.
\end{proof}

\section{Proof of Theorem~\ref{thm:main_3}}

In this section, we modify Aubin's argument for the classical Yamabe problem to prove Theorem~\ref{thm:main_3}. In the sequel, we assume that $W\not\equiv0$.

\subsection{The Weyl tensor}

The $(0,4)$-type Weyl tensor is defined by
\eq{
    W = {\rm Riem} - \frac{1}{n-2}\Big( {\rm Ric}-\frac{R}{n}g \Big)\KN g - \frac{R}{2n(n-1)}g\KN g,
}
where ${\rm Riem}$ is the Riemann tensor, ${\rm Ric}$ is the Ricci tensor, $R$ is the scalar curvature, and $\KN$ is the Kulkarni--Nomizu product. It is well known that the Weyl tensor $W$ is conformally covariant, and shares the same symmetries as the Riemann tensor, and in addition is trace-free, namely $W_{ijik}=0$. For $n\leq3$, $W$ vanishes identically. For $n\geq4$, $W\equiv0$ if and only if $(M,g)$ is locally conformally flat.

Given a point $x_0\in M$ with $W(x_0)\neq0$, it is well known that one can choose a conformal metric $\tilde{g}\in[g]$ with $\tilde{g}(x_0)=g(x_0)$, such that in the normal coordinates in a small geodesic ball $B_r^M(x_0)$,
\eq{\label{eq:25.1}
    \tilde{g}_{ij} = \delta_{ij} - \frac{1}{3}W_{ikjl}(x_0)y^ky^l + O(\abs{y}^3), \qquad \rdV_{\tilde{g}} = (1+O(\abs{y}^3))\rd y.
}

We now compute the first variation of the norm of differential forms. For $p$-forms $\alpha,\beta$, we denote the contraction by
\eq{\label{eq:25.1.1}
    (\alpha\cdot\beta)_{ij}\coloneq \<e_i\ip\alpha,e_j\ip\beta\> = \frac{1}{(p-1)!}\alpha_{ik_2\cdots k_p}\beta_{jk_2\cdots k_p}.
}

\begin{lemma}
Given any $p$-form $\xi$ and any symmetric $2$-tensor $h$, we have
\eq{\label{eq:25.2}
    \frac{\rd}{\rd t}\Big|_{t=0} \abs{\xi}_{\delta+th}^{\frac{2n}{n+1}} = -\frac{n}{n+1}\abs{\xi}_\delta^{-\frac{2}{n+1}} h^{ij}(\xi\cdot\xi)_{ij}.
}
\end{lemma}
\begin{proof}
Note that
\eq{\label{eq:26}
    \abs{\xi}_{\delta+th}^2 = \frac{1}{p!}(\delta+th)^{i_1j_1}\cdots (\delta+th)^{i_pj_p}\xi_{i_1\cdots i_p}\xi_{j_1\cdots j_p}.
}
Since
\eq{
    \frac{\rd}{\rd t}\Big|_{t=0} (\delta+th)^{ij} = -h^{ij},
}
differentiating \eqref{eq:26} gives
\eq{
    \frac{\rd}{\rd t}\Big|_{t=0} \abs{\xi}_{\delta+th}^2 = -\frac{p}{p!}h^{i_1j_1}\delta^{i_2j_2}\cdots\delta^{i_pj_p}\xi_{i_1\cdots i_p}\xi_{j_1\cdots j_p} = -h^{ij}(\xi\cdot\xi)_{ij}.
}
The claim then follows by the chain rule.
\end{proof}

\begin{corollary}
In the small geodesic ball $B_r^M(x_0)$, we have
\eq{\label{eq:27}
    \abs{\xi}_{\tilde{g}}^{\frac{2n}{n+1}} = \abs{\xi}_\delta^{\frac{2n}{n+1}} + \frac{n}{3(n+1)}W_{ikjl}(x_0)\abs{\xi}^{-\frac{2}{n+1}}(\xi\cdot\xi)^{ij}y^ky^l + O(\abs{y}^3\abs{\xi}^{\frac{2n}{n+1}}) + O(\abs{y}^4\abs{\xi}^{\frac{2n}{n+1}}).
}
\end{corollary}
\begin{proof}
It follows from \eqref{eq:25.1} and \eqref{eq:25.2} by choosing $h_{ij}=-\frac{1}{3}W_{ikjl}(x_0)y^ky^l+O(\abs{y}^3)$.
\end{proof}

\begin{lemma}\label{lem5.3}
Let $\beta_\epsilon$ be defined as in \eqref{eq:1.6}.
Then we have the second-order expansion
\eq{
    J_{\tilde{g}}(\beta_\epsilon) = \frac{n+1}{2}\omega_n^{\frac{1}{n}}\Big( 1-\frac{c_W(v_0)}{3\omega_n}\epsilon^2 + o(\epsilon^2) \Big),
}
where 
\eq{\label{eq:27.1}
    c_W(v_0)\coloneqq W_{ikjl}(x_0)\int_{\R^n}\frac{2}{1+\abs{y}^2}(v_0\cdot v_0)^{ij}y^ky^l \rd y.
}
\end{lemma}
\begin{proof}
Choose $\xi=\rd\beta_\epsilon$ in \eqref{eq:27}. Using \eqref{eq:25.1}, and integrating on $B_{r/2}$ (on which $\eta\equiv1$), we have
\eq{\label{eq:28}
\int_{B_{r/2}^M(x_0)} \abs{\rd\beta_\epsilon}_{\tilde{g}}^{\frac{2n}{n+1}} \rdV_{\tilde{g}} = \int_{B_{r/2}} \abs{\rd v_\epsilon}_\delta^{\frac{2n}{n+1}} \rd y + \frac{n}{3(n+1)}W_{ikjl}(x_0) \int_{B_{r/2}} \abs{\rd v_\epsilon}^{-\frac{2}{n+1}}(\rd v_\epsilon\cdot\rd v_\epsilon)^{ij}y^ky^l \rd y + o(\epsilon^2).
}
Using \eqref{eq:1.1} and \eqref{eq:8} we see that (for sufficiently small $\epsilon>0$)
\eq{\label{eq:29}
    \int_{B_{r/2}} \abs{\rd v_\epsilon}_\delta^{\frac{2n}{n+1}} \rd y = \int_{\R^n} \abs{\rd v_0}_\delta^{\frac{2n}{n+1}} \rd y + o(\epsilon^2)
}
and
\eq{\label{eq:30}
    \int_{B_{r/2}} \abs{\rd v_\epsilon}_\delta^{\frac{2n}{n+1}}\abs{y}^2 \rd y = \epsilon^2\int_{\R^n} \abs{\rd v_0}_\delta^{\frac{2n}{n+1}}\abs{y}^2 \rd y + o(\epsilon^2).
}
Using \eqref{eq:3}, \eqref{eq:4}, and \eqref{eq:5} we see that
\eq{\label{eq:31}
    \int_{B_r^M(x_0)\backslash B_{r/2}^M(x_0)} \abs{\rd \beta_\epsilon}_{\tilde{g}}^{\frac{2n}{n+1}} \rdV_{\tilde{g}} = o(\epsilon^2).
}
Combining \eqref{eq:28}, \eqref{eq:29}, \eqref{eq:30}, and \eqref{eq:31} gives 
\eq{\label{eq:31.1}
    \int_M \abs{\rd\beta_\epsilon}_{\tilde{g}}^{\frac{2n}{n+1}} \rdV_{\tilde{g}} = \int_{\R^n} \abs{\rd v_0}_\delta^{\frac{2n}{n+1}} \rd y + \frac{n\epsilon^2}{3(n+1)}W_{ikjl}(x_0) \int_{\R^n} \abs{\rd v_0}^{-\frac{2}{n+1}}(\rd v_0\cdot\rd v_0)^{ij}y^ky^l \rd y + o(\epsilon^2).
}
Recall \eqref{eq:0.2.1} that
\eq{\label{eq:32}
    \abs{v_0}_\delta = \Big( \frac{2}{1+\abs{y}^2} \Big)^{\frac{n-1}{2}}.
}
Using \eqref{conformal_change}, \eqref{eq:0.1}, \eqref{eq:0.2}, and \eqref{eq:0.3} we have
\eq{\label{eq:33}
    \rd v_0 = \Psi^*\rd\alpha_0 = \frac{n+1}{2}\Psi^*(*_{\S^n}\alpha_0) = \frac{n+1}{2} *_{\big(\frac{2}{1+\abs{y}^2}\big)^2\delta}v_0 = \frac{n+1}{2}\cdot\frac{2}{1+\abs{y}^2}*_\delta v_0.
}
Hence
\eq{\label{eq:33.1}
    \int_{\R^n} \abs{\rd v_0}_\delta^{\frac{2n}{n+1}} \rd y = \Big(\frac{n+1}{2}\Big)^{\frac{2n}{n+1}} \int_{\R^n} \Big(\frac{2}{1+\abs{y}^2}\Big)^n \rd y = \Big(\frac{n+1}{2}\Big)^{\frac{2n}{n+1}}\omega_n.
}
By definition \eqref{eq:25.1.1} we have
\eq{\label{eq:34}
    (*_\delta v_0\cdot *_\delta v_0)_{ij} &= \<e_i\ip*_\delta v_0,e_j\ip*_\delta v_0\> = \<e^i\w v_0, e^j\w v_0\> = \<v_0,e_i\ip e^j\w v_0\> \\
    &= \<v_0,\delta_{ij}v_0-e^j\w e_i\ip v_0\> = \abs{v_0}_\delta^2\delta_{ij} - \<e_j\ip v_0,e_i\ip v_0\> = \abs{v_0}_\delta^2\delta_{ij} - (v_0\cdot v_0)_{ij},
}
where we used Proposition~\ref{basic_form} (1)(2) and (3). Combining \eqref{eq:32}, \eqref{eq:33}, and \eqref{eq:34} gives
\eq{
    \abs{\rd v_0}^{-\frac{2}{n+1}}(\rd v_0\cdot\rd v_0)^{ij} = \Big(\frac{n+1}{2}\Big)^{\frac{2n}{n+1}} \frac{2}{1+\abs{y}^2} \big( \abs{v_0}_\delta^2\delta^{ij} - (v_0\cdot v_0)^{ij} \big).
}
By the trace-free property of the Weyl tensor, we have
\eq{
    W_{ikjl}(x_0)\delta^{ij}=0.
}
Hence
\eq{\label{eq:35}
    W_{ikjl}(x_0)\int_{\R^n} \abs{\rd v_0}^{-\frac{2}{n+1}}(\rd v_0\cdot\rd v_0)^{ij}y^ky^l \rd y = -\Big(\frac{n+1}{2}\Big)^{\frac{2n}{n+1}}c_W(v_0).
}
Combining \eqref{eq:31.1}, \eqref{eq:33.1}, and \eqref{eq:35} gives
\eq{
    \int_M \abs{\rd\beta_\epsilon}_{\tilde{g}}^{\frac{2n}{n+1}} \rdV_{\tilde{g}} = \Big(\frac{n+1}{2}\Big)^{\frac{2n}{n+1}}\omega_n - \Big(\frac{n+1}{2}\Big)^{\frac{2n}{n+1}}\frac{n\epsilon^2}{3(n+1)}c_W(v_0) + o(\epsilon^2).
}
Hence
\eq{\label{eq:36}
    \Big( \int_M \abs{\rd\beta_\epsilon}_{\tilde{g}}^{\frac{2n}{n+1}} \rdV_{\tilde{g}} \Big)^{\frac{n+1}{n}} = \Big(\frac{n+1}{2}\Big)^2\omega_n^{\frac{n+1}{n}} \Big( 1 - \frac{c_W(v_0)}{3\omega_n}\epsilon^2 + o(\epsilon^2) \Big).
}
For the denominator, recall \eqref{eq:2.1} that
\eq{\label{eq:37}
    \int_M \beta_\epsilon\w\rd\beta_\epsilon = \int_{\R^n} v_0\w\rd v_0 + o(\epsilon^2) = \frac{n+1}{2}\omega_n + o(\epsilon^2).
}
The Lemma follows from \eqref{eq:36} and \eqref{eq:37}.
\end{proof}

\subsection{A distinguished Euclidean minimizer}\label{sec5.2}

We now show that there exists a Euclidean minimizer $v_0$ such that $c_W(v_0)>0$.

First, we claim that each oriented $\frac{n-1}{2}$-dimensional subspace $V\subset T_{x_0}M\cong\R^n$ induces a Euclidean minimizer. In fact, given $V$, let $\omega_V$ be the unit volume form of $V$. Define an $(n+1)/2$-form
\eq{\label{eq:37.1}
    \Omega_V \coloneqq e^\flat\w\omega_V + \sigma *_{\R^n}\omega_V, 
}
where $e^\flat$ is a constant $1$-form, which is dual to the unit normal vector $e$ at $x_0$, and $\sigma\in\{\pm1\}$ to be determined. It is clear that
\eq{
    *_{\R^{n+1}}(e^\flat\w\omega_V) = \pm *_{\R^n}\omega_V,
}
and 
hence $\Omega_V$ is a self-dual or anti self-dual $\frac{n+1}{2}$-form in $\R^{n+1}$, determined by the choice of $\sigma$. Set
\eq{
    \alpha_V \coloneqq \iota^*(x\ip\Omega_V),
}
where $\iota:\S^n\to\R^{n+1}$ is the inclusion, and $x$ is the position vector. Since $\Omega_V$ is a constant $\frac{n+1}{2}$-form, using Proposition~\ref{basic_form} (4) we have
\eq{
    \rd\alpha_V = \iota^*\rd(x\ip\Omega_V) = \iota^*\mathcal{L}_x\Omega_V = \frac{n+1}{2}\iota^*\Omega_V,
}
and hence
\eq{
    \curl\alpha_V = \frac{n+1}{2}*_{\S^n}\iota^*\Omega_V = \pm\frac{n+1}{2}\alpha_V. 
}
Therefore, we may choose $\sigma\in\{\pm1\}$ such that $\curl\alpha_V = \frac{n+1}{2}\alpha_V$. Moreover, it is clear that
\eq{
    \abs{\alpha_V}\equiv1, \qquad \abs{\rd\alpha_V}\equiv\frac{n+1}{2}.
}
Hence $\alpha_V$ is a spherical minimizer. Let $v_V$ be the associated Euclidean minimizer as in \eqref{eq:0.1}, namely
\eq{
    v_V = \Psi^*\alpha_V.
}
For a more explicit form of $v_V$, see the proof of the following Lemma.
Denote by $P_V$ the orthogonal projection from $ T_{x_0}M\cong\R^n$ onto $V$, and $Q_V\coloneqq \rid-P_V$. We now prove a key identity.

\begin{lemma}\label{lem5.4}
For each oriented $\frac{n-1}{2}$-dimensional subspace $V\subset T_{x_0}M\cong\R^n$, we have
\eq{
    c_W(v_V) = \frac{6\omega_n}{n(n+1)}W(P_V,P_V).
}
Here we define $W(A,B)\coloneqq W_{ikjl}A^{ij}B^{kl}$.
\end{lemma}
\begin{proof}
Given any $y\in\R^n$, recall \eqref{eq:0.1.1} that
\eq{\label{eq:38}
    \Psi(y) = e + \frac{2}{1+\abs{y}^2}(y-e).
}
Hence for any $w\in T_y\R^n\cong\R^n$,
\eq{\label{eq:39}
    \rd\Psi_y(w) = \frac{2}{1+\abs{y}^2}w - \frac{4}{(1+\abs{y}^2)^2}\<y,w\>(y-e) = \frac{2}{1+\abs{y}^2}\Big( w - \frac{2}{1+\abs{y}^2}\<y-e,w\>(y-e) \Big),
}
where we used $e\perp w$. Note that $\abs{y-e}^2=1+\abs{y}^2$, hence
\eq{
    R_y \coloneqq \rid - \frac{2}{1+\abs{y}^2}(y-e)\otimes(y-e) = \rid - 2\frac{y-e}{\abs{y-e}}\otimes\frac{y-e}{\abs{y-e}}
}
is the reflection with respect to the hyperplane $(y-e)^\perp$. Now \eqref{eq:38} and \eqref{eq:39} become
\eq{
    \Psi(y)=R_y(e), \qquad \rd\Psi_y = \frac{2}{1+\abs{y}^2}R_y|_{\R^n}.
}
For any $w_1,\cdots,w_{\frac{n-1}{2}}\in\R^n$, we have
\eq{
    v_V|_y(w_1,\cdots,w_{\frac{n-1}{2}}) &= (\Psi^*\alpha_V)|_y(w_1,\cdots,w_{\frac{n-1}{2}}) = \alpha_V|_{\Psi(y)}(\rd\Psi_y(w_1),\cdots,\rd\Psi_y(w_{\frac{n-1}{2}})) \\
    &= \Omega_V|_{\Psi(y)}(\Psi(y),\rd\Psi_y(w_1),\cdots,\rd\Psi_y(w_{\frac{n-1}{2}})) \\
    &= \Big( \frac{2}{1+\abs{y}^2} \Big)^{\frac{n-1}{2}} \Omega_V|_{\Psi(y)}(R_y(e),R_y(w_1),\cdots,R_y(w_{\frac{n-1}{2}})).
}
Hence
\eq{\label{eq:40}
    v_V = \Big( \frac{2}{1+\abs{y}^2} \Big)^{\frac{n-1}{2}} e\ip(R_y^*\Omega_V).
}
Using \eqref{eq:37.1} we have
\eq{\label{eq:41}
    R_y^*\Omega_V &= \Omega_V - \frac{2}{1+\abs{y}^2}(y-e)^\flat\w(y-e)\ip\Omega_V \\
    &= \Omega_V - \frac{2}{1+\abs{y}^2}(y-e)^\flat\w(y-e)\ip(e^\flat\w\omega_V + \sigma *_{\R^n}\omega_V) \\
    &= \Omega_V - \frac{2}{1+\abs{y}^2}(y-e)^\flat\w\big( -\omega_V - e^\flat\w(y-e)\ip\omega_V + \sigma (y-e)\ip*_{\R^n}\omega_V \big) \\
    &= \Omega_V - \frac{2}{1+\abs{y}^2}(y-e)^\flat\w\big( -\omega_V - e^\flat\w P_V(y)\ip\omega_V + \sigma Q_V(y)\ip*_{\R^n}\omega_V \big).
}
Moreover,
\eq{\label{eq:42}
    e\ip\Omega_V &= e\ip(e^\flat\w\omega_V + \sigma *_{\R^n}\omega_V) = \omega_V - e^\flat\w e\ip\omega_V = \omega_V,
}
and
\eq{\label{eq:43}
    &e\ip(y-e)^\flat\w\big( -\omega_V - e^\flat\w P_V(y)\ip\omega_V + \sigma Q_V(y)\ip*_{\R^n}\omega_V \big) \\
    =& \,\omega_V + e^\flat\w P_V(y)\ip\omega_V - \sigma Q_V(y)\ip*_{\R^n}\omega_V \\&+ (y-e)^\flat\w e\ip \big( \omega_V + e^\flat\w P_V(y)\ip\omega_V - \sigma Q_V(y)\ip*_{\R^n}\omega_V \big) \\
    =& \,\omega_V + e^\flat\w P_V(y)\ip\omega_V - \sigma Q_V(y)\ip*_{\R^n}\omega_V + (y-e)^\flat\w P_V(y)\ip\omega_V \\
    =&\, \omega_V - \sigma Q_V(y)\ip*_{\R^n}\omega_V + y^\flat\w P_V(y)\ip\omega_V,
}
where we used Proposition~\ref{basic_form} (3). Combining \eqref{eq:41}, \eqref{eq:42}, and \eqref{eq:43} gives
\eq{\label{eq:44}
    &e\ip R_y^*\Omega_V = \omega_V - \frac{2}{1+\abs{y}^2} \big( \omega_V - \sigma Q_V(y)\ip*_{\R^n}\omega_V + y^\flat\w P_V(y)\ip\omega_V \big) \\
    &= \Big( 1 - \frac{2}{1+\abs{y}^2} \Big)\omega_V + \frac{2}{1+\abs{y}^2}\sigma Q_V(y)\ip*_{\R^n}\omega_V - \frac{2}{1+\abs{y}^2}(P_V(y)+Q_V(y))^\flat\w P_V(y)\ip\omega_V \\
    &= \Big( 1 - \frac{2}{1+\abs{y}^2}(1+\abs{P_V(y)}^2) \Big)\omega_V + \frac{2}{1+\abs{y}^2}\sigma Q_V(y)\ip*_{\R^n}\omega_V - \frac{2}{1+\abs{y}^2}Q_V(y)^\flat\w P_V(y)\ip\omega_V \\
    &\eqcolon \beta_1 + \beta_2 + \beta_3.
}

We next compute $v_V\cdot v_V$. By continuity, it suffices to consider the case $P_V(y)\neq0$ and $Q_V(y)\neq0$. After choosing a suitable frame, we have
\eq{
    \omega_V = e^1\w\cdots\w e^{\frac{n-1}{2}} \qquad\hbox{with}\quad e_1 = \frac{P_V(y)}{\abs{P_V(y)}}, \quad e_{\frac{n+1}{2}} = \frac{Q_V(y)}{\abs{Q_V(y)}}.
}
For simplicity, we denote
\eq{\label{eq:44.1}
    \rho_1\coloneqq \frac{2}{1+\abs{y}^2}, \qquad \rho_2\coloneqq 1 - \frac{2}{1+\abs{y}^2}(1+\abs{P_V(y)}^2).
}
It follows that
\eq{
    \beta_1\cdot\beta_1 &= \rho_2^2P_V, \\
    \beta_2\cdot\beta_2 &= \rho_1^2 \big( \abs{Q_V(y)}^2Q_V - Q_V(y)\otimes Q_V(y) \big), \\
    \beta_3\cdot\beta_3 &= \rho_1^2 \Big\{ \abs{P_V(y)}^2 Q_V(y)\otimes Q_V(y) + \abs{Q_V(y)}^2 \big(\abs{P_V(y)}^2P_V - P_V(y)\otimes P_V(y)\big) \Big\},
}
and
\eq{
    &\beta_1\cdot\beta_2=\beta_2\cdot\beta_1=0, \qquad \beta_2\cdot\beta_3=\beta_3\cdot\beta_2=0, \\
    &\beta_1\cdot\beta_3=-\rho_1\rho_2P_V(y)\otimes Q_V(y), \qquad \beta_3\cdot\beta_1=-\rho_1\rho_2Q_V(y)\otimes P_V(y).
}
Inserting into \eqref{eq:44} gives
\eq{\label{eq:45}
    S(y) &\coloneqq (e\ip R_y^*\Omega_V)\cdot(e\ip R_y^*\Omega_V) \\
    &= \big(\rho_2^2 + \rho_1^2\abs{P_V(y)}^2\abs{Q_V(y)}^2\big)P_V + \rho_1^2\abs{Q_V(y)}^2Q_V \\
    &\quad + \rho_1^2\big(\abs{P_V(y)}^2-1\big)Q_V(y)\otimes Q_V(y) - \rho_1^2\abs{Q_V(y)}^2P_V(y)\otimes P_V(y) \\
    &\quad -\rho_1\rho_2\big(P_V(y)\otimes Q_V(y)+Q_V(y)\otimes P_V(y)\big).
}
Note that
\eq{
    W(a\otimes b, c\otimes d) = W_{ikjl}(a\otimes b)^{ij}(c\otimes d)^{kl} = W_{ikjl}a^ib^jc^kd^l = W(a,c,b,d).
}
By the trace-free property of the Weyl tensor, we have
\eq{
    W(Q_V,y\otimes y) &= -W(P_V,y\otimes y), \\
    W(P_V(y)\otimes P_V(y),y\otimes y) &= W(P_V(y)\otimes P_V(y),Q_V(y)\otimes Q_V(y)), \\
    W(Q_V(y)\otimes Q_V(y),y\otimes y) &= W(P_V(y)\otimes P_V(y),Q_V(y)\otimes Q_V(y)), \\
    W(P_V(y)\otimes Q_V(y)+Q_V(y)\otimes P_V(y),y\otimes y) &= -2W(P_V(y)\otimes P_V(y),Q_V(y)\otimes Q_V(y)).
}
Combining with \eqref{eq:45} gives
\eq{\label{eq:46}
    W(S(y),y\otimes y) &= \Big( \rho_2^2 + \rho_1^2\abs{Q_V(y)}^2(\abs{P_V(y)}^2-1) \Big)W(P_V,y\otimes y) \\
    &\quad + \Big( \rho_1^2(\abs{P_V(y)}^2-\abs{Q_V(y)}^2-1) + 2\rho_1\rho_2 \Big)W(P_V(y)\otimes P_V(y),Q_V(y)\otimes Q_V(y)).
}
Using \eqref{eq:44.1} and $\abs{y}^2=\abs{P_V(y)}^2+\abs{Q_V(y)}^2$, we have
\eq{
    \rho_2^2 + \rho_1^2\abs{Q_V(y)}^2(\abs{P_V(y)}^2-1) = 1 - 2\rho_1^2\abs{Q_V(y)}^2,
}
and
\eq{
    \rho_1^2(\abs{P_V(y)}^2-\abs{Q_V(y)}^2-1) + 2\rho_1\rho_2 = -2\rho_1^2.
}
Inserting into \eqref{eq:46} and using \eqref{eq:40} gives
\eq{\label{eq:47}
    &\Big( \frac{1+\abs{y}^2}{2} \Big)^{n-1}W(v_V\cdot v_V,y\otimes y) = W(S(y),y\otimes y) \\
    &= \big(1 - 2\rho_1^2\abs{Q_V(y)}^2\big)W(P_V,y\otimes y) - 2\rho_1^2W(P_V(y)\otimes P_V(y),Q_V(y)\otimes Q_V(y)).
}

Finally, we compute $c_W(v_V)$. By definition \eqref{eq:27.1} we have
\eq{\label{eq:48}
    c_W(v_V) &= W_{ikjl}(x_0)\int_{\R^n}\frac{2}{1+\abs{y}^2}(v_V\cdot v_V)^{ij}y^ky^l \rd y = \int_{\R^n}\Big(\frac{2}{1+\abs{y}^2}\Big)^n W(S(y),y\otimes y) \rd y \\
    &= \int_0^\infty \Big(\frac{2}{1+r^2}\Big)^n \Big( \int_{\S^{n-1}} W(S(r\theta),r\theta\otimes r\theta) \rdV_{\S^{n-1}}(\theta) \Big) r^{n-1} \rd r.
}
Using
\eq{
    \fint_{\S^{n-1}} \theta_i\theta_j\,\rdV_{\S^{n-1}} = \frac{1}{n}\delta_{ij}, \qquad \fint_{\S^{n-1}} \theta_i\theta_j\theta_k\theta_l \,\rdV_{\S^{n-1}} = \frac{\delta_{ij}\delta_{kl}+\delta_{ik}\delta_{jl}+\delta_{il}\delta_{jk}}{n(n+2)},
}
we see that
\eq{\label{eq:49}
    \fint_{\S^{n-1}} W(P_V,r\theta\otimes r\theta) = W_{ikjl}(x_0)(P_V)^{ij}r^2\fint_{\S^{n-1}} \theta^k\theta^l = 0,   
}
and
\eq{\label{eq:50}
    \fint_{\S^{n-1}} \abs{Q_V(r\theta)}^2W(P_V,r\theta\otimes r\theta) &= W_{ikjl}(x_0)(P_V)^{ij}(Q_V)_{ab}r^4\fint_{\S^{n-1}}\theta^a\theta^b\theta^k\theta^l \\
    &= \frac{2r^4}{n(n+2)}W(P_V,Q_V) = -\frac{2r^4}{n(n+2)}W(P_V,P_V).
}
Moreover,
\eq{\label{eq:51}
    &\fint_{\S^{n-1}} W(P_V(r\theta)\otimes P_V(r\theta),Q_V(r\theta)\otimes Q_V(r\theta)) \\
    &= W_{ikjl}(x_0)(P_V)^i_a(P_V)^j_b(Q_V)^k_c(Q_V)^l_d r^4 \fint_{\S^{n-1}} \theta^a\theta^b\theta^c\theta^d \\
    &= \frac{r^4}{n(n+2)}W_{ikjl}(x_0) \big( (P_V)^{ij}(Q_V)^{kl} + (P_VQ_V)^{ik}(P_VQ_V)^{jl} + (P_VQ_V)^{il}(P_VQ_V)^{jk} \big) \\
    &= \frac{r^4}{n(n+2)}W_{ikjl}(x_0)(P_V)^{ij}(Q_V)^{kl} = \frac{r^4}{n(n+2)}W(P_V,Q_V) = -\frac{r^4}{n(n+2)}W(P_V,P_V).
}
Combining \eqref{eq:47}, \eqref{eq:49}, \eqref{eq:50}, and \eqref{eq:51} gives
\eq{
    \fint_{\S^{n-1}} W(S(r\theta),r\theta\otimes r\theta) = \frac{6\rho_1^2r^4}{n(n+2)}W(P_V,P_V).
}
Inserting into \eqref{eq:48} yields
\eq{
    c_W(v_V) &= \frac{6\omega_{n-1}}{n(n+2)}W(P_V,P_V)\int_0^\infty \Big(\frac{2}{1+r^2}\Big)^{n+2}r^{n+3} \rd r = \frac{6\omega_n}{n(n+1)}W(P_V,P_V),
}
where the last step follows by a computation of definite integral. Hence we complete the proof.
\end{proof}

\begin{lemma}\label{lem5.5}
If $W(x_0)\neq0$, then there exists an oriented $\frac{n-1}{2}$-dimensional subspace $V\subset T_{x_0}M\cong\R^n$ such that $W(P_V,P_V)>0$.
\end{lemma}
\begin{proof}
Denote
\eq{
    K_{ij} \coloneqq W_{ijij}(x_0).
}
Since the Weyl tensor is determined by its values on $2$-planes, we may choose coordinates, such that there exists some $K_{12}\neq0$. Since
\eq{
    \sum_jK_{1j} = \sum_jW_{1j1j}(x_0) = 0,
}
without loss of generality we may assume $K_{12}>0$.

Next, we use an induction argument. Assume for $2\leq k< \frac{n-1}{2}$ that there exists an oriented $k$-dimensional subspace $V_k={\rm span}\{e_1,\cdots,e_k\}$ such that
\eq{
    S_k \coloneqq \sum_{1\leq i<j\leq k} K_{ij} > 0.
}
Set
\eq{
    A_k(X,Y) \coloneqq \sum_{a=1}^k W(e_a,X,e_a,Y), \qquad X,Y\in V_k^\perp.
}
Then $A_k$ is a symmetric $2$-tensor on $V_k^\perp$, hence has real eigenvalues. Choose $e_{k+1}\in V_k^\perp$ to be the unit eigenvector with respect to the maximal eigenvalue $\lambda_{{\rm max}}$ of $A_k$. By the trace-free property of the Weyl tensor, we have
\eq{
    {\rm tr}(A_k) = \sum_{b>k}\sum_{a=1}^k K_{ab} = \sum_{a=1}^k\sum_{b>k}K_{ab} = -\sum_{a=1}^k\sum_{b\leq k} K_{ab} = -2S_k.
}
Hence
\eq{
    \lambda_{{\rm max}} \geq \frac{{\rm tr}(A_k)}{n-k} = \frac{-2S_k}{n-k}.
}
By our choice of $e_{k+1}$, we see that
\eq{
    S_{k+1} = S_k + \sum_{a=1}^k K_{a,k+1} = S_k + A_k(e_{k+1},e_{k+1}) = S_k + \lambda_{{\rm  max}} \geq \Big( 1 - \frac{2}{n-k} \Big)S_k > 0, 
}
where we used $k<\frac{n-1}{2}$ and hence $\frac{2}{n-k}<1$. Let $V\coloneqq V_{\frac{n-1}{2}}$. Induction gives
\eq{
    0 < S_{\frac{n-1}{2}} = \sum_{1\leq i<j\leq \frac{n-1}{2}} W_{ijij} = \frac{1}{2}W(P_V,P_V).
}
Hence we complete the proof.
\end{proof}

We now prove Theorem~\ref{thm:main_3}.

\begin{proof}[Proof of Theorem~\ref{thm:main_3}]
It is a consequence of Lemma~\ref{lem5.3}, Lemma~\ref{lem5.4}, and Lemma~\ref{lem5.5}.
\end{proof}

\appendix

\section{Regularity theorem for the curl--Yamabe equation}\label{appendix_regularity}

In this Appendix, we give a complete proof of the regularity in Theorem~\ref{thm:main_2}. The argument is very similar to \cite[Appendix C]{WZ26curl}, where we prove the same result on $\S^n$.

Note that the Euler--Lagrange equation of \eqref{eq:inf} is
\eq{\label{eq:true_E-L}
    \curl\big(\abs{\curl \eta}^{-\frac{2}{n+1}}\curl \eta\big) = \mu\,\curl \eta
}
for some constant $\mu>0$. Set
\eq{\label{eq:A1}
    \alpha \coloneqq \Big( \frac{2\mu}{n+1} \Big)^{\frac{n-1}{2}} \abs{\curl \eta}^{-\frac{2}{n+1}}\curl \eta,
}
then \eqref{eq:true_E-L} becomes
\eq{\label{eq:A2}
    \curl\alpha = \frac{n+1}{2}\abs{\alpha}^{\frac{2}{n-1}}\alpha,
}
which is exactly \eqref{eq:Yamabe}.

In Section~\ref{sec4} we have already proved that there exists a minimizer $\eta\in W^{1,\frac{2n}{n+1}}(M)$, which must solve \eqref{eq:true_E-L}. Therefore, by \eqref{eq:A1} and \eqref{eq:A2} we see that $\alpha\in L^{\frac{2n}{n-1}}$, which implies $\rd\alpha\in L^{\frac{2n}{n+1}}$. Moreover, one easily checks that $J(\alpha)=J(\eta)$, hence $\alpha$ is also a minimizer.

\begin{theorem}
Let $\alpha\in L^{\frac{2n}{n-1}}(M)$ be an $\frac{n-1}{2}$-form that solves \eqref{eq:A2}. Then
\eq{
    \alpha\in C^{0,\tau}(M) \cap W^{1,2}(M)\cap C^\infty(\{\alpha\neq0\})
}
for some $\tau>0$.
\end{theorem}
\begin{proof} The proof is the same as that in \cite[Appendix C]{WZ26curl}. For convenience of the reader, we repeat it here.

\smallskip
\noindent\emph{Step 1: H\"older continuity.}
Fix a small ball $B_r\subset M$. The averaged Poincar\'e homotopy estimate (see \cite[Proposition 4.1]{IL93}) implies that there exists a bounded linear operator $T_r$ such that
\eq{\label{eq:A4}
    \alpha = T_r(\rd\alpha) + \rd(T_r\alpha) \eqcolon \beta_r + \rd\phi_r,
}
and
\eq{
    \norm{\beta_r}_{L^{\frac{2n}{n-1}}(B_r)} \lesssim \norm{\rd\alpha}_{L^{\frac{2n}{n+1}}(B_r)}
    = \frac{n+1}{2}\norm{\alpha}_{L^{\frac{2n}{n-1}}(B_r)}^{\frac{n+1}{n-1}}.
}
Let $h=\rd\psi$ be the $p$-harmonic replacement of $\rd\phi_r$ in $B_r$, with $p=\frac{2n}{n-1}$; namely, $\psi$ minimizes $\int_{B_r}\abs{\rd\psi}^{p}$ among competitors with $\psi-\phi_r\in W^{1,p}_0(B_r)$. Then
\eq{
    \rd h = 0, \qquad \rd^*\big( \abs{h}^{\frac{2}{n-1}}h \big) = 0,
}
and the interior estimate yields (see e.g. \cite[Theorem~4.1]{Hamburger92})
\eq{\label{C2}
    \sup_{B_{r/2}} \abs{h}^{\frac{2n}{n-1}} \lesssim r^{-n}\int_{B_r}\abs{h}^{\frac{2n}{n-1}}.
}
Testing 
\eq{
    \rd^*\big( \abs{\alpha}^{\frac{2}{n-1}}\alpha \big) = 0, \qquad
    \rd^*\big( \abs{h}^{\frac{2}{n-1}}h \big) = 0
}
against $\phi_r-\psi$ and subtracting, we obtain
\eq{
    \int_{B_r} \<\abs{\alpha}^{\frac{2}{n-1}}\alpha - \abs{h}^{\frac{2}{n-1}}h, \alpha-h\>
    = \int_{B_r}\<\abs{\alpha}^{\frac{2}{n-1}}\alpha - \abs{h}^{\frac{2}{n-1}}h, \beta_r\>.
}
Using the standard monotonicity and growth bounds
\eq{
    \<\abs{\alpha}^{\frac{2}{n-1}}\alpha - \abs{h}^{\frac{2}{n-1}}h, \alpha-h\> \geq c_n\abs{\alpha-h}^{\frac{2n}{n-1}},
}
\eq{
    \Abs{\abs{\alpha}^{\frac{2}{n-1}}\alpha - \abs{h}^{\frac{2}{n-1}}h}
    \leq C_n\big( \abs{\alpha}^{\frac{n+1}{n-1}} + \abs{h}^{\frac{n+1}{n-1}} \big),
} we obtain, by
 H\"older's inequality, 
\eq{\label{C3}
    \int_{B_r}\abs{\alpha-h}^{\frac{2n}{n-1}}
    \lesssim \Big( \norm{\alpha}_{\frac{2n}{n-1}}^{\frac{n+1}{n-1}} + \norm{h}_{\frac{2n}{n-1}}^{\frac{n+1}{n-1}} \Big)
    \norm{\beta_r}_{\frac{2n}{n-1}}.
}
For $r$ sufficiently small we may assume $\int_{B_r}\abs{\alpha}^{\frac{2n}{n-1}}\leq \epsilon$ (with $\epsilon$ to be fixed).  By minimality of $h$ we  have $\norm{h}_{\frac{2n}{n-1}}\lesssim \norm{\alpha}_{\frac{2n}{n-1}}$ on $B_r$, and hence
\eq{
    \int_{B_r}\abs{\alpha-h}^{\frac{2n}{n-1}} \lesssim \Big( \int_{B_r}\abs{\alpha}^{\frac{2n}{n-1}} \Big)^{\frac{n+1}{n}} \leq \epsilon\Big(\int_{B_r}\abs{\alpha}^{\frac{2n}{n-1}} \Big)^{\frac{1}{n}}.
}
Define
\eq{
    \Psi_x(s) \coloneqq \int_{B_s(x)}\abs{\alpha}^{\frac{2n}{n-1}}.
}
Using \eqref{C2},  for $\theta\in\big(0,\frac12\big)$ we get 
\eq{
    \int_{B_{\theta r}(x)}\abs{h}^{\frac{2n}{n-1}} \lesssim \theta^n \Psi_x(r).
}
Combining the last two bounds yields
\eq{
    \Psi_x(\theta r) \leq C\big( \theta^n + \Psi_x(r)^{\frac{1}{n}} \big)\Psi_x(r).
}
Fix $\mu<n$. Choose $\theta\in\big(0,\frac12\big)$ such that $C\theta^n\leq\frac12\theta^\mu$, and then choose $\epsilon>0$ such that $C\epsilon^{1/n}\leq\frac12\theta^\mu$. Then
\eq{\label{C4}
    \Psi_x(\theta r) \leq \theta^\mu\Psi_x(r).
}
Iterating \eqref{C4} yields
\eq{
    \int_{B_r(x)}\abs{\alpha}^{\frac{2n}{n-1}} \leq C(\mu,\alpha)\, r^\mu, \qquad x\in M.
}
We now apply the homogeneous interior H\"older estimate to the replacement $h=h_r$, see e.g. \cite{Hamburger92}, which implies that for some $\sigma>0$,
\eq{
    \inf_c \int_{B_s(x)}\abs{h_r-c}^{\frac{2n}{n-1}} \lesssim (s/r)^{n+\frac{2n}{n-1}\sigma} \int_{B_r(x)}\abs{h_r}^{\frac{2n}{n-1}}
}
for any $s\in(0,r/2]$. Here $c$ is a constant $\frac{n-1}{2}$-form. The previous comparison estimate and Morrey decay give that
\eq{\label{C4.1}
     \inf_c \int_{B_s(x)}\abs{\alpha-c}^{\frac{2n}{n-1}} \lesssim r^{\frac{n+1}{n}\mu} + (s/r)^{n+\frac{2n}{n-1}\sigma} r^\mu.
}
For sufficiently small $s$, we may assume $r=s^a$ for $a\in(\frac{n}{n+1},1)$. Then the right-hand side of \eqref{C4.1} becomes
\eq{
    s^{\frac{n+1}{n}a\mu} + s^{(1-a)(n+\frac{2n}{n-1}\sigma)+a\mu}.
}
Since $a\in(\frac{n}{n+1},1)$, we see that
\eq{
    \lim_{\mu\to n}\frac{n+1}{n}a\mu = (n+1)a > n,
}
and
\eq{
    \lim_{\mu\to n}(1-a)(n+\frac{2n}{n-1}\sigma)+a\mu = n+\frac{2n}{n-1}(1-a)\sigma>n.
}
Hence choosing $\mu$ sufficiently close to $n$ gives
\eq{
    \inf_c \int_{B_s(x)}\abs{\alpha-c}^{\frac{2n}{n-1}} \lesssim s^{n+\frac{2n}{n-1}\tau}
}
for some $\tau>0$. Campanato's theorem \cite[Page 183]{Campanato63} then implies $\alpha\in C^{0,\tau}$ for some $\tau>0$.

\smallskip
\noindent\emph{Step 2: $W^{1,2}$ estimate.}
Since $\alpha\in C^{0,\tau}$, the right-hand side of \eqref{eq:A2} is H\"older continuous, and hence so is $\rd\alpha$. Set $\beta\coloneqq \rd^*G(\rd\alpha)$, where $G$ is the Hodge Green operator. It satisfies the elliptic system
\eq{
    \rd\beta = \rd\alpha\in C^{0,\tau}, \qquad \rd^*\beta=0,
}
hence standard Schauder theory implies $\beta\in C^{1,\tau}\subset W^{1,2}$. The Hodge decomposition gives
\eq{
    \alpha = \beta + \chi + \rd\phi, \quad \beta+\chi\in C^{1,\tau}\subset W^{1,2},
}
where $\chi$ is the harmonic component.

To estimate $\rd\phi$ we use a regularization argument. Define
\eq{
    J_\epsilon(\psi) \coloneqq \frac{n-1}{2n}\int_{M}\big( \epsilon^2 + \abs{\beta+\chi+\rd\psi}^2 \big)^{\frac{n}{n-1}},
}
and let $\phi_\epsilon$ be a minimizer subject to $\rd^*\phi_\epsilon=0$ and $\phi_\epsilon\perp$ \{harmonic forms\} (when $n=3$, subject to $\int_{M}\phi_\epsilon=0$). The existence of $\phi_\epsilon$ follows from the direct method. Set $\alpha_\epsilon\coloneqq \beta+\chi+\rd\phi_\epsilon$. The Euler--Lagrange equation is 
\eq{\label{C5}
    \rd^*\Big( (\epsilon^2+\abs{\alpha_\epsilon}^2)^{\frac{1}{n-1}}\alpha_\epsilon \Big) = 0.
}
Since $\rd^*(\beta+\chi)=0$, expanding \eqref{C5} yields
\eq{\label{C6}
    \abs{\rd^*\rd\phi_\epsilon}
    \leq \frac{2}{n-1}\abs{\nabla\alpha_\epsilon}
    \leq \frac{{2}}{n-1}\big(\abs{\nabla\rd\phi_\epsilon}+\abs{\nabla(\beta+\chi)}\big).
}
The Weitzenb\"ock formula gives
\eq{\label{C7}
    \norm{\nabla\rd\phi_\epsilon}_2^2 \leq \norm{\rd^*\rd\phi_\epsilon}_2^2 + C(M,g)\norm{\rd\phi_\epsilon}_2^2.
}
Combining \eqref{C6} and \eqref{C7} we obtain
\eq{\label{eq:A5}
    \norm{\nabla\rd\phi_\epsilon}_2
    \leq \frac{2}{n-1}\norm{\nabla\rd\phi_\epsilon}_2 + \frac{2}{n-1}\norm{\nabla(\beta+\chi)}_2 + C(M,g)\norm{\rd\phi_\epsilon}_2.
}
If $n\geq 5$ this implies
\eq{\label{eq:A6}
    \norm{\nabla\rd\phi_\epsilon}_2 \leq \frac{2}{n-3}\norm{\nabla(\beta+\chi)}_2 + C(M,g)\norm{\rd\phi_\epsilon}_2.
}
Since $J_\epsilon(\phi_\epsilon)\leq J_\epsilon(\phi)$, we see that
\eq{
    \sup_{0<\epsilon\leq1}\norm{\alpha_\epsilon}_2 \lesssim \sup_{0<\epsilon\leq1}\norm{\alpha_\epsilon}_{\frac{2n}{n-1}} < \infty.
}
Hence Minkowski's inequality gives
\eq{
    \sup_{0<\epsilon\leq1}\norm{\rd\phi_\epsilon}_2 \leq \sup_{0<\epsilon\leq1}\norm{\alpha_\epsilon}_2 + \norm{\beta+\chi}_2 < \infty.
}
Therefore, the right-hand side of \eqref{eq:A6} is uniformly bounded. In order to see that $\rd\phi_\epsilon$ converges to $\rd\phi$, we set
\eq{
    F_\epsilon(u) \coloneqq \frac{n-1}{2n}\int_M \big( \epsilon^2 + \abs{u}^2 \big)^{\frac{n}{n-1}}, \qquad u\in\Omega^{\frac{n-1}{2}}(M).
}
By \eqref{eq:A2} we have $\rd^*(\abs{\alpha}^{\frac{2}{n-1}}\alpha)=0$. Hence for any $u\in\beta+\chi+\rd\Omega^{\frac{n-3}{2}}$, we have
\eq{
    u-\alpha=\rd\psi
}
for some $\psi\in\Omega^{\frac{n-3}{2}}$. It follows that
\eq{
    \int_M \< \abs{\alpha}^{\frac{2}{n-1}}\alpha,u-\alpha\> = \int_M \< \abs{\alpha}^{\frac{2}{n-1}}\alpha,\rd\psi\> = 0.
}
Since $F_0$ is convex, we have
\eq{
    F_0(u) \geq F_0(\alpha) + \int_M \< \abs{\alpha}^{\frac{2}{n-1}}\alpha,u-\alpha\> = F_0(\alpha),
}
and hence $\alpha$ is the unique minimizer of $F_0$ in the affine space $\beta+\chi+\rd\Omega^{\frac{n-3}{2}}$. Now, if $\alpha_\epsilon\rightharpoonup\bar{\alpha}$, then the weak lower semi-continuity gives
\eq{
    F_0(\bar{\alpha}) \leq \liminf_{\epsilon\to0}F_0(\alpha_\epsilon) \leq \limsup_{\epsilon\to0}F_\epsilon(\alpha_\epsilon) \leq \lim_{\epsilon\to0}F_\epsilon(\alpha) = F_0(\alpha),
}
where the last inequality holds since $\alpha_\epsilon$ minimizes $F_\epsilon$. Hence the uniqueness of $\alpha$ gives $\bar{\alpha}=\alpha$. Therefore, passing to a subsequence as $\epsilon\to0$ in \eqref{eq:A5}, we see that $\rd\phi_\epsilon\rightharpoonup\rd\phi\in W^{1,2}$, and hence $\alpha \in W^{1,2}$.

If $n=3$, then \eqref{C5} becomes
\eq{
    \rd^*\Big( (\epsilon^2+\abs{\alpha_\epsilon}^2)^{\frac{1}{2}}\alpha_\epsilon \Big) = 0,
}
where $\alpha_\epsilon$ is a $1$-form. A direct algebraic computation shows
\eq{\label{C8}
    \abs{\rd^*\rd\phi_\epsilon} = \abs{\rd^*\alpha_\epsilon} \leq \frac{1}{\sqrt{3}}\abs{\nabla\alpha_\epsilon}.
}
Replacing \eqref{C6} by \eqref{C8} and repeating the same argument again gives $\alpha\in W^{1,2}$.

\smallskip
\noindent\emph{Step 3: higher regularity away from the zero set.}
On $\{\alpha\neq0\}$ the system 
\eq{
    \rd\alpha = \frac{n+1}{2}\abs{\alpha}^{\frac{2}{n-1}}*\alpha, \qquad \rd^*\big( \abs{\alpha}^{\frac{2}{n-1}}\alpha \big) = 0
}
is locally uniformly elliptic, so standard elliptic regularity implies $\alpha\in C^\infty$.
\end{proof}

\printbibliography

\end{document}